\documentclass{amsart}
\usepackage[utf8]{inputenc}
\usepackage{amsmath,amssymb,amsfonts,amsthm,enumerate,booktabs}
\title[Curvatures of left invariant Randers metric]{Curvatures of left invariant Randers metric on the five-dimensional Heisenberg group}
\author{A.~Lengyeln\'e-T\'oth} 
\author{Z.~Kov\'acs}
\date{\today}
\address{Corresponding author: Z.~Kov\'acs, University of Ny\'iregyh\'aza,  4400 Ny\'iregyh\'aza, S\'ost\'oi \'ut 31/b}
\email{kovacs.zoltan@nye.hu}
\usepackage{graphicx}
\subjclass[2010]{53B40}
\keywords{Randers metric, Heisenberg group, Chern--Rund connection, flag curvature}

\newtheorem*{theoremm}{Theorem}

\theoremstyle{definition}
\newcommand{\Span}{\operatorname{span}}
\dedicatory{Dedicated to Professor P\'eter T.~Nagy on the occassion of his seventieth birthday}
\begin{document}
\begin{abstract}
A left invariant $Z$-Randers metric on the five-dimensional Heisenberg group is a left invariant Randers metric with deformation vector from the center of the Heisenberg algebra.
In this note we prove that for every left invariant $Z$-Randers metric on the five-dimensional Heisenberg group there exist flags of strictly negative and there exist flags of strictly positive curvatures.
\end{abstract}
\maketitle

\section{Introduction}
The geometry of any Lie group $N$ with left invariant Riemann metric reflects strongly the algebraic structure of its Lie algebra $\mathcal N$. Many of the results in \cite{MR1296558} illustrate this principle. We mention here just two well-known theorems:
\begin{theoremm}[Milnor \cite{MR0425012}]
If $Z$ belongs to the center of the Lie algebra $\mathcal N$, then for any left invariant metric the inequality for the sectional curvature $K(Z,X)\geq 0$ is satisfied for all $X$.
\end{theoremm}
\begin{theoremm}[Wolf \cite{MR0162206}]
Any nonabelian nilpotent Lie group with left invariant metric must admit both positive and negative sectional curvatures.
\end{theoremm}

 The purpose of this paper is to develop the above results for a special type of five-dimensional Randers spaces. A general study of Berwald-type Randers metric on two-step homogeneous nilmanifolds of dimension five can be found in \cite{MR2807115}, where the author shows that the only space which admits left-invariant Randers metric of Berwald type has three-dimensional center. In that case the author gives explicit formula for the flag curvature, and the sign of the flag curvature is studied. In the present paper we study strictly non-Berwald Randers metrics on the two-step nilmanifold with one-dimensional center (usually called Heisenberg manifold). The three-dimensional case is treated in \cite{MR3231513}. We use basically a local calculus. In \cite{MR3231513} we use the Berwald-Mo\'or frame for the computations (see e.g.~\cite{MR0090088}), while here the Homolya--Kowalski basis is used (see~\cite{MR2267683}).

\section{Preliminaries}
Through this paper we use \cite{MR1747675} as a basic reference for foundations of Finsler geometry.
A Finsler manifold $(N,F)$ is a differentiable manifold $N$ equipped with a Finsler metric $F\colon TN\to \mathbf{R}$.
The Finsler geometry counterpart of the Riemannian sectional curvature is the \emph{flag curvature}, which can be introduced by considering the osculating Riemannian metric
\begin{equation}
\langle X,Y\rangle_W=\left.\frac{1}{2}\frac{\partial^2}{\partial s\partial t}\right|_{s=t=0}F^2(W+sX+tY),
\end{equation}
where $X,Y\in T_xN$, and $W\in T_xN\setminus\{0\}$. A flag $\sigma(W,X)$ consists of the flag pole $W$ and a two-dimensional subspace spanned by $W$ and a nonzero transverse vector $X$. The \emph{flag curvature} for $X\in T_xN$ is defined by
\begin{equation}
K(\sigma(W,X))=K(W,X)=\frac{\langle R^W(X,W)W,X\rangle_W}{\|W\|_W^2\|X\|_W^2-\langle X,W\rangle_W^2},
\label{eq:flag}
\end{equation}
where 
\[
R^W(X,Y)Z=\nabla^W_X\nabla^W_YZ-\nabla^W_Y\nabla^W_XZ-\nabla^W_{[X,Y]}Z
\]
is the curvature of the Chern-Rund connection $\nabla^W$ for $F$. The Chern-Rund connection for the nowhere vanishing vector field $W$ is the torsion free, almost metric affine connection
$\nabla^W\colon \mathfrak{X}(N)\times\mathfrak{X}(N)\to\mathfrak{X}(N)$, 
defined by the generalized Koszul formula
\begin{align}\label{eq:CR}
2\langle{\nabla^W_XY},{Z}\rangle_W&=
X\langle{Y},{Z}\rangle_W+
Y\langle{Z},{X}\rangle_W-
Z\langle{X},{Y}\rangle_W+\\
&\quad +\langle{[X,Y]},{Z}\rangle_W-
\langle{[Y,Z]},{X}\rangle_W+
\langle{[Z,X]},{Y}\rangle_W-\nonumber\\
&\quad-2\langle{\nabla^W_XW},{Y},{Z}\rangle_W-
2\langle{\nabla^W_YW},{Z},{X}\rangle_W+2\langle{\nabla^W_ZW},{X},{Y}\rangle_W,
\nonumber
\end{align}
where 
\[
\langle{X},{Y},{Z}\rangle_W
=\left.\frac{1}{4}\frac{\partial^3}{\partial r\partial s\partial t}
\right|_{r,s,t=0}F^2
\left(
W+rX+sY+tZ
\right)
\]
is the $(0,3)$-type Cartan tensor (\cite{MR2132661}).
`Almost metric' here means
\[
X\langle{Y},{Z}\rangle_W=\langle{\nabla^W_XY},{Z}\rangle_W
+\langle{Y},{\nabla^W_XZ}\rangle_W+
2\langle{\nabla^W_XW},{Y},{Z}\rangle_W.
\]

Hereafter let $N$ be the five-dimensional Heisenberg group, which is up to isomorphism, the only two-step nilpotent Lie group with a 1-dimensional center (\cite{MR2267683}). Let $\mathcal N$ denote the five-dimensional real Lie algebra of $N$, with center $\mathcal Z=\operatorname{span} Z$, spanned by the element $Z$.  We assume that $\mathcal N$ is equipped with the Euclidean scalar product $\langle,\rangle$, and suppose that $\|Z\|=1$.

For $X_0\in\mathcal N $ with property \ $\|X_0\|<1$ the function
\begin{equation}\label{eq:83gcf}
f\colon\mathcal N\to\mathbf R,\ X\mapsto f(X)=
\sqrt{\langle X,X\rangle}+\langle{X_0},{X}\rangle
\end{equation}
defines a Minkowski functional on $\mathcal N$; therefore, it can be extended to a \emph{left-invariant Randers type Finsler metric} $F$ on the Lie group $N$ of $\mathcal N$ by left translations. From now on elements of $\mathcal N$ are regarded as left invariant vector fields on $N$. For left invariant vector fields the first three terms of the right hand side of \eqref{eq:CR} vanish. 

Excluding the case $X_0=0$, the remaining Randers metrics are non-Riemannian \cite[p.~283]{MR1747675}. In this paper $X_0=\xi Z$ for a real number $0<\xi<1$. This choice gives a geometric relationship between the Lie algebra and the Randers metric, and we call such type of Randers metric \emph{$Z$-Randers metric}.  Moreover, the condition $Z\in\mathcal{Z}$ guarantees that the Randers metric is not Berwald. Namely, the Randers metric based on \eqref{eq:83gcf} is Berwald if and only if $X_0$ is parallel with respect to the Levi-Civita connection of $\langle,\rangle$ (see e.g.~\cite[Theorem 3.1.4.1.]{MR1273129}). The Levi-Civita connection has the form
\[
2\langle{\nabla_XY},{Z}\rangle=
\langle{[X,Y]},{Z}\rangle-
\langle{[Y,Z]},{X}\rangle+
\langle{[Z,X]},{Y}\rangle,\ X,Y,Z\in \mathcal{N}.
\]
Thus, for all $U\in \mathcal{N}:\ \nabla_UX_0=0$ if and only if
\[
\forall U\in \mathcal{N}\ \forall V\in\mathcal{N}:\
\langle[U,V],X_0\rangle=0.
\]
Consequently, if $X_0\in\mathcal{Z}$, then the Randers metric is not Berwald.

The osculating scalar product can be calculated from the Euclidean scalar product by 
\begin{multline}\label{eq:scalar}
\langle {U},{V}\rangle_W=\langle {U},{V}\rangle+\langle {X_0},{U}\rangle\langle {X_0},{V}\rangle
-\langle {X_0},{W}\rangle\langle {W},{U}\rangle\langle {W},{V}\rangle
\\
+
\langle {X_0},{U}\rangle\langle {W},{V}\rangle+
\langle {X_0},{W}\rangle\langle {U},{V}\rangle+
\langle {X_0},{V}\rangle\langle {W},{U}\rangle,
\end{multline}
where $W\in\mathcal N$ and $\langle {W},{W}\rangle=1$; moreover, the $(0,3)$-type symmetric Cartan tensor is
\begin{multline}\label{eq:cartan}
\langle {U},{V},{X}\rangle_{W}=\frac{1}{2}\sum_{[U,V,X]}\left\{
\langle{X_0},{W}\rangle\langle{W},{U}\rangle\langle{W},{V}\rangle\langle{W},{X}\rangle\right.\\
-\left.\langle{X_0},{W}\rangle\langle{X},{V}\rangle\langle{U},{W}\rangle-
\langle{X_0},{X}\rangle\langle{W},{V}\rangle\langle{W},{U}\rangle+
\langle{X_0},{U}\rangle\langle{X},{V}\rangle
\right\},
\end{multline}
see \cite{MR3231513}. 

\section{The main theorem}
\begin{theoremm}
For every left invariant $Z$-Randers metric on the five-dimensional Heisenberg group there exist flags of strictly negative and there exist flags of strictly positive curvatures.
\end{theoremm}
\begin{proof}
Table~\ref{tab:iouzgfv} outlines the proof, where 
we summarized the flag curvature of the $Z$-Randers metric for special flags $\sigma(W,X)$.
\begin{table}
\begin{tabular}{llll}
\toprule
&flag pole $W$&transverse vector $X$&$K(X,W)$\\[1ex]
&&($X\neq W$)&\\[1ex]
\midrule
(1.1)&$W\in\mathcal{Z}$&$X\in\Span(e_1,e_2)$&$\frac{\lambda^2}{4}>0$\\[1ex]
(1.2)&&$X\in\Span(e_3,e_4)$&$\frac{\mu^2}{4}>0$\\[1ex]
\midrule
(2.1)&$W\in\Span(e_1,e_2)$&$X\in\mathcal{Z}$&$\frac{1-\xi^2}{4}\lambda^2>0$\\[1ex]
(2.2)&                        &$X\in\Span(e_1,e_2)$&$\frac{\xi^2-3}{4}\lambda^2<0$\\[1ex]
(2.3)&                        &$X\in\Span(e_3,e_4)$&$\frac{\mu^2-\lambda^2}{4}\xi^2\leq0$\\[1ex]
                        \midrule
(3.1)&$W\in\Span(e_3,e_4)$&$X\in\mathcal{Z}$&$\frac{1-\xi^2}{4}\mu^2>0$\\[1ex]      
(3.2)&                        &$X\in\Span(e_1,e_2)$&$\frac{\lambda^2-\mu^2}{4}\xi^2\geq0$\\[1ex]
(3.3)&                        &$X\in\Span(e_3,e_4)$&$\frac{\xi^2-3}{4}\mu^2<0$\\[1ex]
\bottomrule
\end{tabular}
\caption{Flag curvatures for special flags}
\label{tab:iouzgfv}
\end{table}
In Table~\ref{tab:iouzgfv} we use the Homolya--Kowalski basis.
In \cite{MR2267683} the authors construct an orthonormal basis $(e_1,e_2,e_3,e_4,Z)$ in $\mathcal N$ such that
\begin{equation}
[e_1,e_2]=-[e_2,e_1]=\lambda Z,\ [e_3,e_4]=-[e_4,e_3]=\mu Z,\
\lambda\geq\mu>0,
\end{equation}
and all the other Lie brackets are zero. The inner product here is the Euclidean scalar product $\langle,\rangle$.

We calculate the Chern--Rund connection from the metric using method described in \cite[Theorem 3.10]{MR2132661} directly. Also, this method is used in \cite{MR3231513} for the three-dimensional case. In that paper a more detailed description of the algorithm can be found. We record in 
Tables~\ref{tab:ubi}--\ref{tab:wqidhf7} the explicit formul\ae{} for the Chern--Rund connection, restricted to the demand of the proof.

\paragraph{\bf Cases (1.1) and (1.2)} If $W\in\mathcal{Z}$ then \eqref{eq:cartan} gives $\langle{U},{V},{X}\rangle_Z=0$ for all $U$, $V$, $X$, and from \eqref{eq:scalar} we get
\begin{align*}
\langle{U},{V}\rangle_Z&=(1+\xi)\langle{U},{V}\rangle,\ U,V\in\mathcal{V}\\
\langle{U},{V}\rangle_Z&=0,\ U\in\mathcal{V},\ V\in\mathcal{Z}\\
\langle{U},{V}\rangle_Z&=(1+\xi)(\langle{U},{V}\rangle+
\xi\langle{Z},{U}\rangle\langle{Z},{V}\rangle),\ U,V\in\mathcal{Z}.
\end{align*}
The generalized Koszul formula \eqref{eq:CR} simplifies to
\begin{align}\label{eq:CRuzfuz}
2\langle{\nabla^Z_UV},{X}\rangle_Z&=
\langle{[U,V]},{X}\rangle_Z-
\langle{[V,X]},{U}\rangle_Z+
\langle{[X,U]},{V}\rangle_Z.
\end{align}
From \eqref{eq:CRuzfuz} it is easy to determine the components of the Chern--Rund connection in the basis $(e_1,e_2,e_3,e_4,e_5=Z/(\xi+1))$, see Table~\ref{tab:ubi}.
\begin{table}
\begin{tabular}{rrrrrr}
\toprule
&$e_1$&$e_2$&$e_3$&$e_4$&$e_5$\\[1ex]
\midrule
$\nabla_{e_1}$&0&$\frac{\lambda}{2}Z$&0&0&$-\frac{\lambda}{2}e_2$\\[1ex]
$\nabla_{e_2}$&$-\frac{\lambda}{2}Z$&0&0&0&$\frac{\lambda}{2}e_1$\\[1ex]
$\nabla_{e_3}$&0&0&0&$\frac{\mu}{2}Z$&$-\frac{\mu}{2}e_4$\\[1ex]
$\nabla_{e_4}$&0&0&$-\frac{\mu}{2}Z$&0&$\frac{\mu}{2}e_3$\\[1ex]
$\nabla_{e_5}$&$-\frac{\lambda}{2}e_2$&$\frac{\lambda}{2}e_1$&
$-\frac{\mu}{2}e_4$&$\frac{\mu}{2}e_3$&0\\[1ex]
\bottomrule
\end{tabular}
\caption{Local components of the Chern--Rund connection, $W=Z$}
\label{tab:ubi}
\end{table}
Substituting expressions from Table~\ref{tab:ubi} into \eqref{eq:flag} we get
\[
K(Z,e_1)=\frac{\lambda^2}{4},\
K(Z,e_2)=\frac{\lambda^2}{4},\
K(Z,e_3)=\frac{\mu^2}{4},\
K(Z,e_4)=\frac{\mu^2}{4},
\]
which gives cases (1.1) and (1.2).

In what follows, $W=w_1 e_1+w_2 e_2+w_3 e_3+ w_4 e_4\in\mathcal{V}$, i.e.\ $W$ is from the Euclidean orthogonal complement of the center.

From \eqref{eq:scalar} we get in this case
\begin{align}\
\langle X,Y\rangle_W&=\langle X,Y\rangle+\xi^2\langle Z,X\rangle\langle Z,Y\rangle\
(X,Y\in\mathcal{Z}),\ \text{spec. }\langle Z,Z\rangle_W=1+\xi^2\\
\langle X,Y\rangle_W&=\langle X,Y\rangle\ (X,Y\in\mathcal{V}),\ \text{spec. }\langle W,W\rangle_W=1\\
\langle X,Y\rangle_W&=\xi\langle Z,Y\rangle\langle W,X\rangle\
(X\in\mathcal{V},\ Y\in\mathcal{Z}),\ \text{spec. }\langle Z,W\rangle_W=\xi.\label{eq.uzgub}
\end{align}
Moreover, for the Cartan tensor we have
\begin{equation}\label{eq:cartan34}
\langle {U},{V},{X}\rangle_{W}=\frac{1}{2}\sum_{[U,V,X]}\left\{
-
\langle{X_0},{X}\rangle\langle{W},{V}\rangle\langle{W},{U}\rangle+
\langle{X_0},{U}\rangle\langle{X},{V}\rangle
\right\}.
\end{equation}

Applying the Gram\,--\,Schmidt process we get that
\begin{equation}
(e_1,e_2,e_3,e_4,e_5=Z-\xi W)
\end{equation}
is an orthonormal basis w.r.t.\ the osculating scalar product $\langle,\rangle_W$.

 Let
%\[
%W^\perp=\begin{pmatrix}
%w_1& w_2& w_3& w_4
%\end{pmatrix}
%\begin{pmatrix}
%0&-\lambda&0&0\\
%\lambda&0&0&0\\
%0&0&0&-\mu\\
%0&0&\mu&0
%\end{pmatrix}
%\begin{pmatrix}
%e_1\\e_2\\e_3\\e_4
%\end{pmatrix},
%\]
%i.e.\
\begin{equation}\label{eq:kjjzvb}
W^\perp=
\lambda w_2e_1-\lambda w_1e_2+\mu w_4e_3-\mu w_3 e_4.
\end{equation}
For $W^\perp$ we have
\begin{equation}
\langle{W},{W^\perp}\rangle=\langle{W},{W^\perp}\rangle_W=0.
\end{equation}
Moreover, the generalized Koszul formula \eqref{eq:CR} becomes
\begin{align*}
2\langle{\nabla^W_WW},{e_i}\rangle_W&=
2\langle{[e_i,W]},{W}\rangle_W=
2\xi\langle{Z},{[e_i,W]}\rangle=
\begin{cases}
\phantom{-}2\xi w_2\lambda,& i=1\\
-2\xi w_1\lambda,& i=2\\
\phantom{-}2\xi w_4\mu,&i=3\\
-2\xi w_3\mu,&i=4.
\end{cases}
\end{align*}
Using \eqref{eq:CR} once more
\[
2\langle{\nabla_W^WW},{e_5}\rangle_W=2\langle{[e_5,W]},{W}\rangle_W=0,
\]
hence
\begin{equation}
\nabla^W_WW=\xi W^\perp.
\end{equation}

\paragraph{\bf Case (2.1).} Using Table~\ref{tab:kuhzvx} we get
\begin{align*}
R(Z,W)W&=\nabla^W_Z\nabla^W_WW-\nabla^W_W\nabla^W_ZW=\xi\nabla^W_ZW^\perp-\frac{1}{2}\nabla^W_WW^\perp\\
&=\frac{\lambda^2}{4}(\xi^2-1)\left(\xi W-Z\right).
\end{align*}
It follows that
\begin{align*}
K(W,Z)&=\frac{\langle R(Z,W)W,Z\rangle_W}{\|W\|_W\|Z\|_W-\langle Z,W\rangle_W^2}\\
&=\frac{\lambda^2}{4}\cdot\frac{\xi^2-1}{(1+\xi^2)-\xi^2}\left(\xi^2-(1+\xi^2)\right)=\frac{\lambda^2}{4}(1-\xi^2).
\end{align*}
\paragraph{\bf Case (2.2).}
While $\sigma(W,X)=\sigma(W,W^\perp)$ in this case, there is no loss of generality in choosing $W^\perp$ as the transverse edge. The required components of the Chern--Rund connection are computed from the generalized Koszul formula, and they are in Table~\ref{tab:kuhzvx}. 
\begin{align*}
R^W(W^\perp,W)W&=\nabla^W_{W^\perp}\nabla^W_WW-
\nabla^W_W\nabla^W_{W^\perp}W-\nabla^W_{[W^\perp,W]}W\\
&=\xi\nabla^W_{W^\perp}{W^\perp}
-\frac{1}{2}\lambda^2\left(\nabla^W_WZ-\xi\nabla^W_WW\right)-
\lambda^2\nabla^W_ZW
=-\frac{1}{4}\xi\lambda^2W^\perp,
\end{align*}
hence
\[
K(W,W^\perp)=\frac{1}{4}\lambda^2(\xi^2-3)\frac{\|W^\perp\|_W^2}%
{\|W\|_W^2\|W^\perp\|_W^2}=\frac{1}{4}\lambda^2(\xi^2-3).
\]
\begin{table}
\begin{tabular}{lrrr}
\toprule
&$W$&$W^\perp$&$Z$\\
\midrule
$\nabla_W^W$&$\xi W^\perp$&$-\frac{1}{2}\lambda^2(\xi W+Z)$&$\frac{1}{2}W^\perp$\\[1ex]
$\nabla^W_{W^\perp}$&$\frac{1}{2}\lambda^2(Z-\xi W)$&$-\frac{1}{4}\xi\lambda^2 W^\perp$&$\frac{1}{4}\lambda^2\left((\xi^2-2)W-\xi Z\right)$\\[1ex]
$\nabla^W_Z$&$\frac{1}{2}W^\perp$&$\frac{1}{4}\lambda^2\left((\xi^2-2)W-\xi Z\right)$&$\frac{1}{4}\xi W^\perp$\\[1ex]
\bottomrule
\end{tabular}
\caption{$W\in\Span(e_1,e_2)$}
\label{tab:kuhzvx}
\end{table}
\paragraph{\bf Case (2.3).}
\begin{table}
\begin{tabular}{rrr}
\toprule
&$W$&$W^\perp$\\[1ex]
\midrule
$\nabla^W_{e_3}$&$-\frac{1}{2}\mu\xi e_4$&$-\frac{1}{4}\xi\lambda^2e_3$\\[1ex]
$\nabla^W_{e_4}$&$\frac{1}{2}\mu\xi e_3$&$-\frac{1}{4}\xi\lambda^2e_4$\\[1ex]
\bottomrule
\end{tabular}
\caption{$W\in\Span(e_1,e_2)$}
\label{tab:wqidhf}
\end{table}
Since $[e_3,W]=0$ in this case,
\begin{align*}
R^W(e_3,W)W&=\nabla^W_{e_3}\nabla^W_WW-\nabla^W_W\nabla^W_{e_3}W\\
&=\xi\nabla^W_{e_3}W^\perp+\frac{1}{2}\xi\mu\nabla^W_We_4
=\frac{1}{4}\xi^2(\mu^2-\lambda^2)e_3.
\end{align*}
Consequently,
\[
K(W,e_3)=\frac{\langle R^W(e_3,W)W,e_3\rangle_W}%
{\|e_3\|_W\|W\|_W-\langle e_3,W\rangle_W}=\frac{1}{4}\xi^2(\mu^2-\lambda^2).
\]
Analogously, $R^W(e_4,W)W=\frac{1}{4}\xi^2(\mu^2-\lambda^2)e_4$ and $K(W,e_4)=\frac{1}{4}\xi^2(\mu^2-\lambda^2)$. Hence, for all $X\in\Span(e_3,e_4)$ we have $K(W,X)=\frac{1}{4}\xi^2(\mu^2-\lambda^2)$.

Proof of statements in the rows (3.1), (3.2) and (3.3) of 
Table~\ref{tab:iouzgfv} are completely analogous, and we give only the form of the Chern--Rund connection, see Tables~\ref{tab:uhzgzuh} and~\ref{tab:wqidhf7}. 
\end{proof}

\begin{table}[ht!]
\begin{tabular}{lrrr}
\toprule
&$W$&$W^\perp$&$Z$\\
\midrule
$\nabla_W^W$&$\xi W^\perp$&$-\frac{1}{2}\mu^2(\xi W+Z)$&$\frac{1}{2}W^\perp$\\[1ex]
$\nabla^W_{W^\perp}$&$\frac{1}{2}\mu^2(Z-\xi W)$&$-\frac{1}{4}\xi\mu^2 W^\perp$&$\frac{1}{4}\mu^2\left((\xi^2-2)W-\xi Z\right)$\\[1ex]
$\nabla^W_Z$&$\frac{1}{2}W^\perp$&$\frac{1}{4}\mu^2\left((\xi^2-2)W-\xi Z\right)$&$\frac{1}{4}\xi W^\perp$\\[1ex]
\bottomrule
\end{tabular}
\caption{$W\in\Span(e_3,e_4)$}
\label{tab:uhzgzuh}
\end{table}
\begin{table}[ht!]
\begin{tabular}{rrr}
\toprule
&$W$&$W^\perp$\\[1ex]
\midrule
$\nabla^W_{e_1}$&$-\frac{1}{2}\lambda\xi e_2$&$-\frac{1}{4}\xi\mu^2e_1$\\[1ex]
$\nabla^W_{e_2}$&$\frac{1}{2}\lambda\xi e_1$&$-\frac{1}{4}\xi\mu^2e_2$\\[1ex]
\bottomrule
\end{tabular}
\caption{$W\in\Span(e_3,e_4)$}
\label{tab:wqidhf7}
\end{table}

%\nocite{*}
%\bibliographystyle{plain}
%\bibliography{references}
\end{document}